\documentclass[12pt]{amsart}
\usepackage{pstricks}
\usepackage{euscript}
\def\lambdabar{\bar{\lambda}}
\def\P{\mathbb{P}}
\def\Q{\mathbb{Q}}
\def\R{\mathbb{R}}
\def\Z{\mathbb{Z}}
\def\F{\mathbb{F}}
\def\tensor{\otimes}
\let\iso\cong
\let\cong\equiv
\def\curlyO{\EuScript{O}}
\def\curlyV{\EuScript{V}}
\def\ybar{\bar{y}}
\def\Isom{\mathop{\rm Isom}\nolimits}

\def\PSL{{\rm PSL}}
\def\PGL{{\rm PGL}}
\def\SL{{\rm SL}}
\def\GL{{\rm GL}}
\def\PO{{\rm PO}}
\def\PU{{\rm PU}}
\def\a{\alpha}
\def\c{\gamma}
\def\shat{\hat{s}}
\def\sset{\subseteq}
\def\Zhalf{\Z[\frac{1}{2}]}
\def\Lhalf{L[\frac{1}{2}]}
\def\Mhalf{M[\frac{1}{2}]}
\def\curlyOhalf{\curlyO[\frac{1}{2}]}
\def\G{\Gamma}
\def\PG{P\G}
\def\GMu{\G_{\!M}}
\def\GAK{\G_{\!L}}
\def\PGMu{P\GMu}
\def\PGAK{P\GAK}

\def\B{\mathcal{B}}

\def\generate#1{\langle #1\rangle}
\def\spanof#1{\langle #1\rangle}
\def\hip#1#2{\langle #1|#2\rangle}
\def\bighip#1#2{\bigl\langle #1\bigm|#2\bigr\rangle}
\def\set#1#2{\{#1\,|\,#2\}}

\def\biggset#1#2{\biggl\{#1\,\biggm|\,#2\biggr\}}

\newtheorem{theorem}{Theorem}
\newtheorem{lemma}[theorem]{Lemma}
\newtheorem{corollary}[theorem]{Corollary}

\theoremstyle{remark}

\newtheorem*{remarks}{Remarks}
\numberwithin{theorem}{section}
\numberwithin{equation}{section}
\numberwithin{figure}{section}

\begin{document}

\title{The densest lattices in $\PGL_3(\Q_2)$}
\author{Daniel Allcock}
\address{Department of Mathematics\\University of Texas, Austin}
\email{allcock@math.utexas.edu}
\urladdr{http://www.math.utexas.edu/\textasciitilde allcock}
\author{Fumiharu Kato}
\address{Department of Mathematics\\Kumamoto University}
\email{kato@sci.kumamoto-u.ac.jp}
\urladdr{http://www.sci.kumamoto-u.ac.jp/\textasciitilde kato/index-e.html}
\date{October 25, 2011}

\begin{abstract}
We find the smallest possible covolume for lattices in $\PGL_3(\Q_2)$,
show that there are exactly two lattices with this covolume, and
describe them explicitly. They are commensurable, and one of them
appeared in Mumford's construction of his fake projective plane.  We
also discuss a new $2$-adic uniformization of another fake projective
plane.
\end{abstract}

\maketitle

{\it Correction added September 24, 2016:\/} The proof of lemma~\ref{lem-passage-between-PGL3-and-SL3}
contains an error; even the statement of the lemma is wrong. But this
does not affect the rest of the paper.  Here are the details.  It is
not true that ``central extensions of finite groups by torsion-free
groups are trivial''; for example, $\beta_M\in\GL_3\Q_2$ from
section~\ref{sec-Gamma-L-and-Gamma-M} generates a subgroup $\Z/3$ of $\PGL_3\Q_2$, and this
$\Z/3$ has no lift to $\SL_3\Q_2$.  What is true is that
$\SL_3\Q_2\to\PSL_3\Q_2$ is an isomorphism whose image has index~$3$
in $\PGL_3\Q_2$.  This follows from the structure of
$\Q_2^*\iso\{\pm1\}\times\Z\times\Z_2$.  Furthermore, if an element of
$\GL_3\Q_2$ represents an element of $\PGL_3\Q_2$ outside
$\PSL_3\Q_2$, then its determinant is not a perfect cube in~$\Q_2$.
By the lattice interpretation of the vertices of Bruhat-Tits
building, it cannot fix any vertex.  This establishes an analogue of
lemma~\ref{lem-passage-between-PGL3-and-SL3} that is sufficient for all applications in the paper:
{\it a finite subgroup of $\PGL_3\Q_2$ is the isomorphic image of a unique
  subgroup of $\SL_3\Q_2$, if it fixes some vertex of
  the Bruhat-Tits building.}

Here is an another way to fix lemma~\ref{lem-passage-between-PGL3-and-SL3}; we omit the proof since it
is not needed to fix the paper. {\it Suppose $G$ is a
  finite subgroup of $\PGL_3\Q_2$.  Then either (i) it is the
  isomorphic image of a unique subgroup of $\SL_3\Q_2$, or (ii)
  $G\iso\Z/3$, or (iii) $G$ is isomorphic to the alternating group
  $A_4$ and conjugate to the group generated by projectivizations of 
  $$
  \begin{pmatrix}
    -&&\\&-&\\&&+
  \end{pmatrix}
  ,
  \begin{pmatrix}
    +&&\\&-&\\&&-
  \end{pmatrix}
  ,\hbox{ \rm and }
  \begin{pmatrix}
    &&\sigma\\1&&\\&1&
  \end{pmatrix}
  $$
  where $\sigma\in\Q_2^*$ is not a perfect cube.}

We now return you to the original paper.

\bigskip
The most famous lattice in the projective group $\PGL_3(\Q_2)$ over
the $2$-adic rational numbers $\Q_2$ is the one Mumford used to
construct his fake projective plane \cite{Mumford}.  Namely, he found
an arithmetic group $\PG_1$ (we call it $\PGMu$) containing a
torsion-free subgroup of index~$21$, such that the algebraic surface
associated to it by the theory of $p$-adic uniformization
\cite{Mustafin-1,Kurihara} is a fake projective plane.  The full
classification of fake projective planes has been obtained recently
\cite{FPP}.  Fake projective planes play an important and interesting
role, not only in the theory of algebraic surfaces and non-archimedean
geometry, but also in the theory of algebraic cycles, especially in
connection with the Bloch conjecture, cf. \cite{Barlow}.

The second author and his collaborators have developed a diagrammatic
calculus \cite{CKK,Kato-calculus} for working with algebraic curves
(including orbifolds) arising from $p$-adic uniformization using
lattices in $\PGL_2$ over a nonarchimedean local field.  It allows one
to read off properties of the curves from the quotient of the
Bruhat-Tits tree and to construct lattices with various properties, or
prove they don't exist.  We hope to develop a higher-dimensional
analogue of this theory, although only glimpses of it are now visible.
Pursuing these glimpses suggested the existence of another lattice
$\PGAK$ with the same covolume as Mumford's, and we were able to prove
that it really does exist.  We show that $\PGAK$ and $\PGMu$ have the
smallest possible covolume in $\PGL_3(\Q_2)$, are the only lattices
with this covolume, and meet each other in a common index~$8$
subgroup.  We also give explicit generators and a geometric
description of their actions on the Bruhat-Tits building.

We discovered a new $2$-adically uniformized fake projective plane by
considering the subgroup of $\PGAK$ corresponding to a Sylow
$2$-subgroup of $\PGL_2(\F_7)$ under a natural surjection
$\PGAK\to\PGL_2(\F_7)$.  More precisely, let $\curlyV$ be the
$\Z_2$-scheme associated to the (torsion-free) kernel $K$ of
$\PGAK\to\PGL_2(\F_7)$ by non-Archimedean uniformization.  Unlike in
Mumford's case, the Sylow $2$-subgroup of $\PGAK/K$ acts quite badly
on $\curlyV$: there are components of the central fiber $\curlyV_0$
with nontrivial pointwise stabilizer.  {\it But it acts freely on the
  generic fiber}, with quotient a fake projective plane defined over
$\Q_2$.

In particular, $\PGAK$ led directly to a way to use discrete
groups with torsion to construct interesting varieties that are
genuinely smooth, not just smooth as orbifolds.  Among other things,
this opens up the possibility of $p$-adic uniformizations of fake
projective planes with $p>2$.  This is unexpected because Mumford's
calculations in \cite{Mumford} show that such uniformization is
possible only if the  residue field is $\F_2$.  Since Mumford only
considered torsion-free groups, our results are compatible with his.

Except for the presence of torsion, our construction parallels
Mumford's.  The algebraic surfaces associated to $\PGAK$ and $\PGMu$
are also surprisingly similar, for example they have the same
configuration of singularities.   Ishida \cite{Ishida} has described
the surface in considerable detail in the $\PGMu$ case.  We intend to
carry out the corresponding analysis for $\PGAK$, as well as the
treatment of the fake projective plane, in a future paper.

Finding densest-possible lattices in Lie groups has a long history,
beginning with Siegel's treatment \cite{Siegel}
of the unique densest lattice in $\PSL_2(\R)$.  Lubotzky
\cite{Lubotzky} found the minimal covolume in $\SL_2(\Q_p)$ and
$\SL_2(\F_q((t)))$, and some lattices realizing it.  With Weigel
\cite{Lubotzky-Weigel} he obtained the complete classification of
(isomorphism classes of) densest lattices in $\SL_2$ over any finite
extension of $\Q_p$.  Golsefidy \cite{Golsefidy} identified the unique
densest lattice in $G\bigl(\F_q((t))\bigr)$, where $G$ is any simply connected Chevelley
group of type $E_6$ or classical type${}\neq A_1$ and $q$ is
neither~$5$ nor a power of $2$ or~$3$.  A generalization of Lubotzky's
result in another direction regards $\SL_2\bigl(\F_q((t))\bigr)$ as
the loop group of $\SL_2(\F_q)$.  This is the simplest infinite-dimensional
Kac-Moody group, having type $\tilde{A}_1$ over $\F_q$.  The
next-simplest Kac-Moody groups correspond to symmetric rank~$2$ Cartan matrices of hyperbolic
type.  The minimal covolumes of lattices in these Kac-Moody groups,
and some lattices realizing them, have been found by Capdeboscq and
Thomas \cite{Capdeboscq-Thomas}.

Lattices in $\PO(n,1)$ and $\PU(n,1)$ present special challenges
because not all of them are arithmetic.  Meyerhoff
\cite{Meyerhoff-noncompact-H3-announcement,Meyerhoff-noncompact-H3-proof}
identified the unique densest non-cocompact lattice in the identity component
$\PO^\circ(3,1)$ of $\PO(3,1)$, and with Gabai and Milley he identified the
densest cocompact lattice
\cite{Gabai-Meyerhoff-Milley-closed-H3-outline,Gabai-Meyerhoff-Milley-closed-H3-details}.
Hild and Kellerhals \cite{Hild-Kellerhals} found the
unique densest non-cocompact lattice in $\PO(4,1)$, and Hild
\cite{Hild} extended this to $\PO(n,1)$ for $n\leq9$.  Among
arithmetic lattices, Belolipetsky
\cite{Belolipetsky-SO(even-1),Belolipetsky-SO(even-1)-addendum} found
the unique densest lattice in $\PO^\circ(n,1)$ for  even~$n$,
in both the cocompact and non-cocompact cases.  With Emery
\cite{Belolipetsky-Emery} he extended
this to the case of odd $n\geq5$. 
Stover \cite{Stover} found the two densest non-cocompact arithmetic
lattices in $\PU(2,1)$.

The first author is grateful to the Clay Mathematics Institute, the
Japan Society for the Promotion of Science and Kyoto University for
their support and hospitality.

\section{Finite and discrete subgroups of $\PGL_3(\Q_2)$}
\label{sec-finite-and-discrete-subgroups}

Throughout this section $V$ is a $3$-dimensional vector space over the
$2$-adic rational numbers $\Q_2$. Our goal is to study the finite
subgroups of $\PGL(V)$ and how they constrain the discrete subgroups.

We will write $F_{21}$ for the Frobenius group of order~$21$ (the
unique nonabelian group of this order), $S_n$ for the symmetric group
on $n$ objects, and sometimes $2$ for $\Z/2$ and $2^n$ for $(\Z/2)^n$.
Also, $L_3(2)$ means the simple group $\PSL_3(2)\iso\PSL_2(7)$ of
order~$168$.  It has three conjugacy classes of maximal subgroups
\cite{ATLAS}: the 
stabilizers of points and lines in $\P^2\F_2$,  isomorphic to $S_4$, and 
the Borel
subgroup of $\PSL_2(7)$, isomorphic to $F_{21}$.

The first step in studying the finite subgroups of $\PGL(V)$ is that
passage between $\PGL$ and $\SL$ is free of complications:

\begin{lemma}
\label{lem-passage-between-PGL3-and-SL3}
Every finite subgroup of $\SL(V)$ maps isomorphically to its image in
$\PGL(V)$, and every finite subgroup of $\PGL(V)$ is the image of a
unique subgroup of $\SL(V)$.
\end{lemma}

\begin{proof}
For the first claim, note that the only scalar in $\SL(V)$ is the
identity, so $\SL(V)\to\PGL(V)$ is injective.
Now let $A$ be a complement to $\{\pm1\}$ in
$\Q_2^*\iso\{\pm1\}\times\Z\times\Z_2$, and let $H\sset\GL(V)$ consist of the
transformations with determinants in $A$. 
The center of $H$ is torsion-free.
Since central extensions of
finite groups by torsion-free groups are trivial, a finite subgroup of
$\PGL(V)$ has a unique lift to $H$.  
The determinant of any element of the lift is a finite-order element of $A$, hence equals~$1$.  We have shown that
every finite subgroup of $\PGL(V)$ has a unique lift to $\SL(V)$.
\end{proof}

If $A$ is an integral domain with fraction field $k$ and $W$ a
$k$-vector space, then an $A$-lattice in $W$ means an $A$-submodule
$L$ with $L\tensor_Ak=W$.  We write $\GL(L)$ for the group of
$A$-module automorphisms of $L$, and $\SL(L)$ for its intersection
with $\SL(W)$.  In this paper $A$ will be the ring of integers $\curlyO$ in
$\Q(\sqrt{-7})$, or $\curlyOhalf$, or the $2$-adic integers $\Z_2$.  In
this section it will always be $\Z_2$, and $L$ means a $\Z_2$-lattice
in $V$.  The following lemma is well-known and holds much more
generally, but we only give the case we need.

\begin{lemma}
\label{lem-lattice-is-sum-of-eigenlattices}
Suppose an involution $g\in\GL(L)$ acts trivially on $L/2L$.  Then $L$
is the direct sum of $g$'s eigenlattices (the intersections of $L$
with $g$'s eigenspaces).
\end{lemma}

\begin{proof}
Write $x_{\pm}$ for the
projection of any $x\in L$ to $g$'s $\pm1$ eigenspaces.  By hypothesis $x\pm
gx$ is $0$ mod $2L$, which is to say  $2x_{\pm}\in2L$.
So $x_{\pm}\in L$, proving the lemma.
\end{proof}

\begin{lemma}
\label{lem-finite-subgroups-of-GL3Z2}
Suppose  $G$ is a
finite subgroup of $\SL(L)$.  Then either
\begin{enumerate}
\item
\label{item-G-acts-faithfully-modulo-2}
$G$ acts faithfully on $L/2L$, or
\item
\label{item-G-is-S4-acting-nonfaithfully}
$(G,L)\iso(S_4,\Z_2^3)$ where $S_4$ is the determinant~$1$ subgroup of
the group of signed permutations, or
\item
$|G|\leq12$.
\label{item-G-is-at-most-12}
\end{enumerate}
\end{lemma}

\begin{proof}
Suppose $G$ is not faithful on $L/2L$.  The kernel of
$\SL(L)\to\SL(L/2L)$ is a pro-$2$-group, so its intersection $K$ with $G$
is a $2$-group.  In fact it is elementary abelian because a
modification of Siegel's argument
\cite[\S39]{Siegel-symplectic-geometry} shows that every element of
$K$ has order $1$ or~$2$.  (If $h\in K$ has order~$4$ then decompose
$L$ as a direct sum of $h^2$'s eigenlattices using
lemma~\ref{lem-lattice-is-sum-of-eigenlattices}.  Restricting to the
$-1$ eigenlattice gives $h^2=-I$ and $h=I+2T$ for some matrix $T$ over
$\Z_2$.  It is easy to see that these are incompatible modulo~$4$.)

So we may choose a basis for $V$ in which
$K$ is diagonal.  Obviously each involution in $K$ negates two
coordinates, and $K$ is $2$ or $2^2$.  If $K\iso2^2$ then $V$ is the
sum of three distinct $1$-dimensional representations of $K$, and we
consider the intersections of $L$ with these
subspaces.  Using lemma~\ref{lem-lattice-is-sum-of-eigenlattices}
twice shows that $L$ is their direct sum.  
Since $G$ normalizes $K$, it permutes these summands.
Because $G$ is finite and the torsion subgroup of $\Q_2^*$ is
$\{\pm1\}$, the $G$-stabilizer of a summand acts on it by $\{\pm1\}$.
It follows easily that $L\iso\Z_2^3$ in a manner identifying $G$ with
a subgroup of the $S_4$ in 
\eqref{item-G-is-S4-acting-nonfaithfully}.  
So either
\eqref{item-G-is-S4-acting-nonfaithfully} or
\eqref{item-G-is-at-most-12} holds.

On the other hand, if $K$ has a single involution, then lemma~\ref{lem-lattice-is-sum-of-eigenlattices}
shows that $L$ is the direct sum of its eigenlattices. Of course,
$G/K$ preserves the images of these sublattices in $L/2L$.  The
$L_3(2)$-stabilizer of a point of $\P^2\F_2$ and a line not containing
it is $S_3$, so $|G|\leq12$.
\end{proof}

$\PGL(V)$ acts on its Bruhat-Tits building $\B$.  We recall from
\cite{BT} that this is the simplicial complex with one vertex for each
lattice in $V$, up to scaling.  Often we speak of ``the'' lattice
associated to a vertex when the scale is unimportant.  Two vertices
are joined if and only if one of the lattices contains the other of
index~$2$ (after scaling).  Whenever three vertices are joined
pairwise by edges, there is a $2$-simplex spanning those edges.
$\PGL(V)$ acts transitively on vertices, with $\PGL(L)$ being the
stabilizer of the vertex corresponding to $L$.  The link of this
vertex is the incidence graph of the points and lines of
$\P(L/2L)\iso\P^2\F_2$, on which $\PGL(L)$ acts as
$\GL(L/2L)$.

This subgroup $\PGL(L)$  is a maximal compact subgroup, and we scale
the Haar measure on $\PGL(V)$ so that this subgroup has mass~$1$.  If
$\PG$ is any discrete subgroup of $\PGL(V)$ then it acts on $\B$ with finite
stabilizers.  For each $\PG$-orbit $\Sigma$ of vertices, let
$n_\Sigma$ be the order of the $\PG$-stabilizer of any member of
$\Sigma$.  One can express $\PG$'s covolume (the Haar measure of
$\PGL(V)/\PG$) as the sum of $1/n_\Sigma$ over all $\PG$-orbits
$\Sigma$.  

A lattice in $\PGL(V)$ means a discrete subgroup of finite covolume.
This double use of ``lattice'' is a standard confusion; we hope
context will make our meaning clear.  Mumford \cite{Mumford}
exhibited a lattice in $\PGL_3(\Q_2)$ that acts transitively on
vertices of $\B$, with stabilizer isomorphic to $F_{21}$.  Its
covolume is $1/21$.  The goal of this paper is to show that this is
the smallest possible covolume, and that there is exactly one other
lattice realizing it.

\begin{lemma}
\label{lem-initial-bounds-on-lattices-in-PGL3Q2}
Suppose $\PG$ is a lattice in $\PGL(V)$ of
covolume${}\leq1/21$.  Then either
\begin{enumerate}
\item
\label{item-every-stabilizer-L32}
every vertex stabilizer is isomorphic to $L_3(2)$ and there
are${}\leq8$ orbits, or
\item
\label{item-transitive-with-stabilizer-S4-or-F21}
$\PG$ acts transitively on vertices, with stabilizer isomorphic to
  $F_{21}$ or $S_4$, or
\item
\label{item-2-orbits-with-stabilizers-L32-and-S4}
$\PG$ acts with two orbits on vertices, with stabilizers isomorphic to $L_3(2)$
  and $S_4$.
\end{enumerate}
\end{lemma}

\begin{proof}
Obviously every vertex stabilizer has order${}\geq21$.  We claim that
every finite subgroup $G$ of $\PGL(V)$ of order${}\geq21$ is
isomorphic to $F_{21}$, $S_4$ or $L_3(2)$.  It suffices by
lemma~\ref{lem-passage-between-PGL3-and-SL3} to prove this with $\PGL$
replaced by $\SL$.  Obviously $G$ preserves some $2$-adic lattice $L$,
and lemma~\ref{lem-finite-subgroups-of-GL3Z2} implies that either
$G\iso S_4$ or else $G$ acts faithfully on $L/2L$.  In the latter case
we use the fact that every subgroup of $L_3(2)$ of order${}\geq21$ is
isomorphic to $F_{21}$, $S_4$ or $L_3(2)$.  Having constrained the
vertex stabilizers in $\PG$ to these three groups, one works out which
sums of $1/21$, $1/24$ and $1/168$ are${}\leq1/21$.
\end{proof}

Some of these possibilities cannot occur.  The key is to understand
the three possible $S_4$-actions on a $\Z_2$-lattice $L$:

\begin{lemma}
\label{lem-3-actions-of-S4}
Suppose $G\sset\SL(L)$ is isomorphic to
$S_4$.  Then the pair $(G,L)$ is isomorphic to exactly one of the
following pairs; in each case $S_4$ acts by the determinant~$1$
subgroup of the group of signed permutations.
\begin{enumerate}
\item
\label{item-S4-0}
$(S_4,L_0:=\Z_2^3)$
\item
\label{item-S4-l}
$\bigl(S_4,L_l:=\set{x\in L_0}{x_1+x_2+x_3\cong0\mod2}\bigr)$
\item
\label{item-S4-p}
$\bigl(S_4,L_p:=\set{x\in L_0}{x_1\cong x_2\cong x_3\mod2}\bigr)$.
\end{enumerate}
\end{lemma}

\noindent
We refer to the three cases as types $0$, $l$ and $p$.  The notation
reflects the fact that $L_l$ and $L_p$ correspond
to a line and point in $\P(L_0/2L_0)\iso\P^2\F_2$ respectively.  We
have already seen type~$0$ in
lemma~\ref{lem-finite-subgroups-of-GL3Z2}.  If a group $G\sset\PGL(V)$
isomorphic to $S_4$ fixes a vertex $v$ of $\B$, then we say that $v$
has type $0$, $p$ or~$l$ according to the type of the action of (the
lift to $\SL(V)$ of) $G$ on the lattice represented by~$v$.
   
\begin{proof}
Consider the sublattice $L_0$ of $L$ spanned by the fixed-point
sublattices of the three involutions in the Klein $4$-group $K_4\sset
G$.  Obviously $L_0$ has type~$0$.  Now, $L$ lies
between $\frac{1}{2}L_0$ and $L_0$, so it corresponds to a
$G$-invariant subspace of $L_0/2L_0$.  Besides the $0$ subspace there
are only two, leading to \eqref{item-S4-l} and \eqref{item-S4-p}.  The
three cases may be distinguished by $[L:L_0]$, which is $1$, $2$ or
$4$ respectively.
\end{proof}

\begin{lemma}
\label{lem-S4-actions-on-neighbors}
Suppose $G\sset\PGL(V)$
is isomorphic to $S_4$ and stabilizes a vertex $v$ of $\B$.
\begin{enumerate}
\item
\label{item-type-0-implies-types-p-and-l}
If $v$ has type $0$ then $G$ stabilizes exactly two neighbors of $v$, which
have types $p$ and $l$ with respect to $G$.
\item
\label{item-type-p-or-l-implies-type-0}
If $v$ has type $p$ or $l$ then $G$ stabilizes exactly one neighbor of $v$,
which has type $0$ with respect to $G$.
\end{enumerate}
\end{lemma}

\begin{proof}
Choosing a $G$-equivariant isomorphism of the lattice represented by
$v$ with one of the models in lemma~\ref{lem-3-actions-of-S4} makes
visible the claimed neighboring lattice(s).  In the proof of that
lemma we saw that a type~$0$ $G$-lattice has exactly two $G$-invariant
neighbors.  Similarly, if $L$ is a $G$-lattice of type not $0$, then
$G$ acts faithfully on $L/2L$ (by lemma~\ref{lem-finite-subgroups-of-GL3Z2}).  Then, since $G\iso
S_4$, $G$ is the stabilizer of a point or line in $\P(L/2L)$.  This
makes it easy to see that $G$ fixes only one neighbor of $L$.
\end{proof}

\begin{lemma}
\label{lem-L32-actions-on-neighbors}
Suppose $\PG$ is a subgroup of $\PGL(V)$ and that the $\PG$-stabilizer of some vertex of
$\B$ is isomorphic to $L_3(2)$.  Then the $\PG$-stabilizer of any
neighboring vertex is isomorphic to $S_4$.
\end{lemma}

\begin{proof}
Suppose $v$ is a vertex with stabilizer $G\iso L_3(2)$, $w$ is any
neighboring vertex, and $L$ and $M$ are the associated lattices.
Write $H$ for the $G$-stabilizer of $w$; it is the stabilizer of a point or
line of $\P(L/2L)$, so it is isomorphic to $S_4$.  Also,
$v$ has type $p$ or $l$ with respect to $H$, since $H$ acts faithfully
on $L/2L$ (indeed all of $G$ does).  By
lemma~\ref{lem-S4-actions-on-neighbors}\eqref{item-type-p-or-l-implies-type-0},
$w$ has type $0$ with respect to $H$.

Now, the full $\PG$-stabilizer of $w$ is finite (otherwise the
$\PG$-stabilizer of $v$ would be infinite), so
lemma~\ref{lem-finite-subgroups-of-GL3Z2} applies to it.  Since it
contains an $S_4$ acting nonfaithfully on $M/2M$, it can be no larger
than $S_4$.
\end{proof}

\begin{lemma}
\label{lem-second-bounds-on-lattices-in-PGL3Q2}
If $\PG$ is a lattice in $\PGL(V)$ of covolume${}\leq1/21$, then its
covolume is exactly $1/21$, and either it acts transitively on the
vertices of $\B$, with stabilizer isomorphic to $F_{21}$, or else it
has two orbits, with stabilizers isomorphic to $L_3(2)$ and $S_4$.
\end{lemma}

\begin{proof}
This amounts to discarding some of the possibilities listed in lemma~\ref{lem-initial-bounds-on-lattices-in-PGL3Q2}.
The case that every vertex stabilizer is isomorphic to $L_3(2)$ is
ruled out by lemma~\ref{lem-L32-actions-on-neighbors}.  And if every vertex stabilizer is
isomorphic to $S_4$ then lemma~\ref{lem-S4-actions-on-neighbors} shows that there are at least
$3$ orbits, ruling out the $S_4$ case of lemma~\ref{lem-initial-bounds-on-lattices-in-PGL3Q2}\eqref{item-transitive-with-stabilizer-S4-or-F21}.
\end{proof}

In section~\ref{sec-Gamma-L-and-Gamma-M} we will exhibit lattices realizing these two
possibilities, and in section~\ref{sec-uniqueness} we will show they are the only
ones. 

\section{The Hermitian $\curlyO$-lattices $L$ and $M$}
\label{sec-L-and-M}

In this section we introduce Hermitian lattices $L$ and $M$ over the
ring of algebraic integers $\curlyO$ in $\Q(\sqrt{-7})$.  In the next
section we will study their isometry groups over $\Zhalf$, which turn
out to be the two densest possible lattices in $\PGL_3(\Q_2)$.  The
construction of $M$ is due to Mumford \cite{Mumford}, and a
description of $L$ appears
without attribution in the ATLAS \cite{ATLAS} entry for $L_3(2)$.  $L$
is unimodular and contains $8$ copies of $M$, while $M$ has
determinant~$7$ and lies in exactly one copy of $L$
(lemma~\ref{lem-superlattices-of-M}).

Let $\lambda$ and $\lambdabar$ be the algebraic integers
$\bigl(-1\pm\sqrt{-7}\bigr)/2$. Let $\curlyO$ be the ring of algebraic
integers in $\Q(\lambda)$, namely $\Z+\lambda\Z$.  Everything about
$\lambda$ and $\lambdabar$ can be derived from the equations
$\lambda+\lambdabar=-1$ and $\lambda\lambdabar=2$.  For example,
multiplying the first by $\lambda$ yields the minimal polynomial
$\lambda^2+\lambda+2=0$.  $\curlyO$ is a Euclidean domain, so its class
group is trivial, so any $\curlyO$-lattice is automatically a free
$\curlyO$-module.  We often regard $\curlyO$ as embedded in $\Z_2$.  There are
two embeddings, and we always choose the one with
$\lambdabar$ mapping to a unit and $\lambda$ to twice a unit.

A Hermitian $\curlyO$-lattice means an $\curlyO$-lattice $L$ equipped with an
$\curlyO$-sesquilinear pairing $\hip{}{}:L\times L\to \Q(\lambda)$, linear in its
first argument and anti-linear in its second, satisfying
$\hip{y}{x}=\overline{\hip{x}{y}}$.  It is common to omit
``Hermitian'', but we will be careful to include it, because the
unadorned word ``lattice'' already has two meanings in this paper.

$L$ is called integral if $\hip{}{}$ is $\curlyO$-valued.  Its
determinant $\det L$ is the determinant of the inner product matrix of
any basis for $L$.  This is a well-defined rational integer since
$\curlyO^*=\{\pm1\}$.  An isometry means an $\curlyO$-module isomorphism
preserving $\hip{}{}$, and we write $\Isom L$ for the group of all
isometries.

The central object of this paper is the Hermitian $\curlyO$-lattice
$$
L:=\biggset{(x_1,x_2,x_3)\in\curlyO^3}{\hbox{%
\begin{tabular}{l}%
\hbox{$x_i\cong x_j$ mod $\lambdabar$ for all $i,j$ and}\\
\hbox{$x_1+x_2+x_3\cong0$ mod $\lambda$}%
\end{tabular}}}
$$ using one-half the standard Hermitian form,
$\hip{x}{y}=\frac{1}{2}\sum x_i\ybar_i$. The norm $x^2$ of a vector
means its inner product with itself.  
We call a norm~$2$ vector a root.
$\Isom L$ contains the signed
permutations, which we call the obvious isometries.  Using them
simplifies every verification below to a few examples.

\begin{lemma}
\label{lem-roots-generate-L}
$L$ is integral and unimodular, its minimal norm is~$2$, and has basis
$(2,0,0)$, $(\lambdabar,\lambdabar,0)$ and $(\lambda,1,1)$.  The full
set of roots consists of their $42$ images under obvious isometries.
\end{lemma}

\begin{proof}
It's easy to see that the listed vectors  lie in $L$.  To see they generate
it, consider an arbitrary element of $L$.  By adding a multiple of
$(\lambda,1,1)$ we may suppose the last coordinate is $0$.  Then the
$x_i\cong x_j$ mod $\lambdabar$ condition shows that $\lambdabar$ divides the
remaining coordinates.  By adding a multiple of $(\lambdabar,\lambdabar,0)$ we
may suppose the second coordinate is also $0$.  Then
$x_1+x_2+x_3\cong0$ mod $\lambda$ says that the first coordinate is
divisible by $\lambda$ as well as $\lambdabar$, hence by~$2$.  So it lies in
the $\curlyO$-span of $(2,0,0)$.  We have shown that the three given roots
form a basis for $L$, and computing their inner product
matrix shows that $L$ is integral with determinant~$1$.

All that remains is to enumerate the vectors of norm${}\leq2$
and see that they are as claimed.  This is easy because the
only possibilities for the components are $0$, $\pm1$, $\pm\lambda$,
$\pm\lambdabar$ and $\pm2$.
\end{proof}

The word ``root'' is usually reserved for vectors negated by
reflections, and our next result justifies our use of the term.

\begin{lemma}
\label{lem-roots-of-L-give-reflections}
$\Isom L$ is transitive on roots and contains the reflection in any root $r$, i.e., the map
$$
x\mapsto x-2\frac{\hip{x}{r}}{r^2}r.
$$
\end{lemma}

\begin{proof}
The reflection negates $r$ and fixes $r^\perp$ pointwise.  It also
preserves $L$, since $\hip{x}{r}\in\curlyO$ and $r^2=2$.  Therefore it is
an isometry of $L$.  To show transitivity on roots, note that
$(\lambda,1,1)$ has inner product $-1$ with $(\lambdabar,\lambdabar,0)$.  It
follows that they are simple roots for a copy of the $A_2$
root system.  Since the Weyl group $W(A_2)\iso S_3$ generated by their
reflections acts transitively on the $6$ roots of $A_2$, these two
roots are equivalent.  The same argument shows that $(\lambda,1,1)$ is
equivalent to $(0,0,-2)$.  Together with obvious isometries, this
proves transitivity.
\end{proof}

\begin{corollary}
\label{cor-Aut-L-is-L32}
$\Isom L$ is isomorphic to $L_3(2)\times\{\pm1\}$ and acts on $L/\lambda L$
as $\GL(L/\lambda L)$.  
\end{corollary}

\begin{proof}
Since $\curlyO^*=\{\pm1\}$, $\Isom L$ is the product of its determinant~$1$
subgroup $\Isom^+\!L$ and $\{\pm1\}$.  Now, $\Isom^+\!L$ has order
divisible by $7$ by transitivity on the 42 roots
(lemma~\ref{lem-roots-of-L-give-reflections}), and also contains $24$
obvious isometries.  So it has order at least~$168$.
Since $L/\lambda L= (L\tensor\Z_2)/(2L\tensor\Z_2)$, 
lemma~\ref{lem-finite-subgroups-of-GL3Z2} shows that $\Isom^+\!L$ injects into
$\GL(L/\lambda L)\iso L_3(2)$ (since cases \eqref{item-G-is-S4-acting-nonfaithfully} and
\eqref{item-G-is-at-most-12} cannot apply).  Since $L_3(2)$ has
order~$168$, the injection is an isomorphism.
\end{proof}

Now we turn to defining a Hermitian $\curlyO$-lattice $M$ which we will
recognize after lemma~\ref{lem-basic-properties-of-M} as a copy of
Mumford's (see \cite[p.~240]{Mumford}).  It turns out that everything
about it is best understood by embedding it in $L$.  So even though
our aim is to understand $M$, we will develop some further properties
of $L$.

We use the term ``frame'' to refer to either a nonzero element of
$L/\lambda L$ or the set of roots mapping to it modulo~$\lambda$.
Since $\Isom L$ acts on $L/\lambda L\iso\F_2^3$ as $\GL(L/\lambda L)$,
it acts transitively on frames, so each frame has $42/7=6$ roots.  The
standard frame means
$\{(\pm2,0,0),\discretionary{}{}{}(0,\pm2,0),\discretionary{}{}{}(0,0,\pm2)\}$.
The language ``frame'' reflects the fact that each frame consists of 3
mutually orthogonal pairs of antipodal vectors.  The stabilizer of the
standard frame is exactly the group of obvious isometries.

\begin{lemma}
\label{lem-norm-3-and-7-vectors-in-L}
$L$ has $56$ norm~$3$ and $336$ norm~$7$ vectors, and $\Isom L$ is
transitive on each set.  
\end{lemma}

\begin{proof}
Every element of $\lambda L$ has even norm, so the norm~$3$ and~$7$
vectors lie outside $\lambda L$.  We will find all $x\in L$ of norms~$3$
and $7$ that represent the standard frame.  Representing the standard frame implies
that $\lambda$ divides $x$'s coordinates.  Now, either $\lambdabar$ divides
all the coordinates or none; if it divides all then so does $2$
and $x^2=\hbox{odd}$ is impossible.  So $\lambdabar$ divides none.
Writing $x=\lambda(a,b,c)$ we have $a,b,c\in\curlyO-\lambdabar\curlyO$ with
$|a|^2+|b|^2+|c|^2=\hbox{$3$ or $7$}$.  The possibilities for
$(a,b,c)$ are very easy to work out, using the fact that the only
elements of $\curlyO-\lambdabar\curlyO$ of norm${}<7$ are $\pm1,\pm\lambda,\pm\lambda^2$.
The result is that there are $8$ (resp.~$48$) vectors of norm~$3$
(resp.~$7$) that represent the standard frame, all of them equivalent
under obvious isometries.  By transitivity on frames, $L$ contains
$7\cdot8$ (resp.\ $7\cdot48$) norm~$3$ (resp.~$7$) vectors and $\Isom
L$ is
transitive on them.
\end{proof}

\begin{remarks}
The proof shows that each norm~$3$ resp.~$7$ vector $x$ has a unique
description as $\lambdabar^{-1}(e+e'+e'')$
resp.\ $\lambdabar^{-1}(e+\lambda e'+\lambda^2 e'')$, where $e$, $e'$,
$e''$ are mutually orthogonal roots of the frame represented by $x$.
The same argument proves transitivity on the~$168$ norm~$5$ vectors,
each of which has a unique description as $\lambdabar^{-1}(e+\lambda
e'+\lambda e'')$.  One can also check that every norm~$4$ (resp.~$6$)
vector lies in just one of $\lambda L$ and $\lambdabar L$, so there
are $42+42$ (resp. $56+56$) of them.  It turns out that $L$ admits an
anti-linear isometry, namely the map $\beta_L$ from
section~\ref{sec-Gamma-L-and-Gamma-M} followed by complex conjugation.
Enlarging $\Isom L$ to include anti-linear isometries gives a group
which is transitive on the vectors of each norm $2,\dots,7$.
\end{remarks}

We write $\theta$ for
$\lambda-\lambdabar=\sqrt{-7}$.   We define $s$ to be the norm
seven vector $(\lambda,-\lambda^2,\lambda^3)$ and
$M:=\set{x\in L}{\hbox{$\hip{x}{s}\cong0$ mod $\theta$}}$.  By
lemma~\ref{lem-norm-3-and-7-vectors-in-L}, using any other norm~$7$
vector in place of $s$ would yield an isometric Hermitian lattice.
To understand $M$ we use the fact that it is the preimage of a
hyperplane in $L/\theta L\iso\F_7^3$.  Since $L$ is unimodular,
$\hip{}{}$ reduces to a nondegenerate symmetric bilinear form on
$L/\theta L$.  Since vectors of norm${}\leq3$ are too close together to be
congruent mod~$\theta$, the roots (resp.\ norm~$3$ vectors) represent
$42$ (resp.~$56$) distinct elements of $L/\theta L$.  These correspond
to the $21$ ``minus'' points (resp. $28$ ``plus'' points) of
$\P(L/\theta L)$, which in ATLAS terminology \cite[p.~xii]{ATLAS} means
the nonisotropic points orthogonal to no (resp. some) isotropic
point.  Lurking behind the scenes here is that $\Isom
L\iso2\times L_3(2)$ has index~$2$ in the full isometry group
$2\times\PGL_2(7)$ of $L/\theta L$.  By the transitivity of $\Isom L$
on the $|\P^1\F_7|\cdot|\F_7^*|=48$ isotropic vectors, each is
represented by $336/48=7$ norm~$7$ vectors.

Now, $M$ is the preimage of $\shat^\perp\sset L/\theta L$,
where the hat means the image mod~$\theta$.  It follows that the
subgroup of $\Isom L$ preserving $M$ is $2\times F_{21}$.  The reason we
chose the sign on the second coordinate of $s$ is so that $M$ is
preserved by the cyclic permutation of coordinates, rather than some
more complicated isometry of order~$3$.  To check this, one just
computes
$$
\bighip{(\lambda,-\lambda^2,\lambda^3)}{(\lambda^3,\lambda,-\lambda^2)}
\cong 0\mod\theta
$$
and uses the fact that in 
a $3$-dimensional nondegenerate inner product space, isotropic vectors are
orthogonal if and only if they are proportional.  

\begin{lemma}
\label{lem-basic-properties-of-M}
$M$ contains no roots of $L$ and exactly $42$ norm~$7$ vectors. It
contains 
exactly $14$ norm~$3$ vectors, namely
the images of 
$$
e_1=(-\lambdabar^2,-\lambdabar,0)
\qquad
e_2=(\lambda,\lambda,\lambda)
\qquad
e_3=(1,\lambda^2,1)
$$
under $\generate{\hbox{cyclic permutation},-1}\iso\Z/6$.  These three
vectors form a basis for $M$, with inner product matrix
$$
\begin{pmatrix}
3&\lambdabar&\lambdabar\\
\lambda&3&\lambdabar\\
\lambda&\lambda&3
\end{pmatrix}.
$$
\end{lemma}

\begin{proof}
Observe that $\shat^\perp\sset L/\theta L$ has $6$ nonzero isotropic elements
and $14$ of each norm $3\cdot(\hbox{a nonzero square in $\F_7$})$.  Therefore
$M$ contains no roots, and since each norm~$3$ element of $L/\theta L$
is represented by exactly one norm~$3$ element of $L$, $M$ has exactly
$14$ norm~$3$ vectors.  It is easy to check that the three displayed
vectors lie in $M$.  Since $M$
is preserved by cyclic permutation, they yield all $14$ norm~$3$
vectors.  

It is easy to check that $\hip{e_i}{e_j}=\lambdabar$ if $i<j$, so the
inner product matrix is as stated.  Since $L/M$ is $1$-dimensional
over $\curlyO/\theta\curlyO\iso\F_7$, we see that $\det M$ is $7$
times $\det L$.  Since the $e_i$
have inner product matrix of determinant~$7$, they form a basis for
$M$.  
\end{proof}

Since Mumford describes his Hermitian lattice in terms of a basis with
this same inner product matrix, our $M$ is a copy of it.  (Once we
suspected Mumford's lattice lay inside $L$, there was only one
candidate for it, and searching for the $e_i$'s realizing his inner
product matrix was easy.  Using the known $2\times F_{21}$ symmetry,
we could without loss take $e_2$ and $e_3$ as stated.  Then there were
just three possibilities for $e_1$.)

Modulo~$\theta$, the matrix is the all $3$'s matrix, hence has
rank~$1$.  The following result is needed for
theorem~\ref{thm-intersection-has-index-8-in-each}, on the index of
$\GAK\cap\GMu$ in $\GAK$ and $\GMu$, but not elsewhere.

\begin{lemma}
\label{lem-superlattices-of-M}
$M$ is a sublattice of exactly~$8$ unimodular Hermitian lattices, one of which
  is $L$ and the rest of which are isometric to $\curlyO^3$.  In
  particular, the isometry group of $M$ is the subgroup $2\times
  F_{21}$ of $\Isom L$ preserving $M$.
\end{lemma}

\begin{proof}
Since $\det M=7$, any unimodular Hermitian superlattice $L'$ contains
it of index~$7$.  So $L'$ corresponds to a $1$-dimensional subspace
$S'$ of $M/\theta M$.  That is, $L'=\spanof{M,\frac{1}{\theta}v}$
where $v\in M$ represents any nonzero element of $S'$.  In order for
$L'$ to be integral, $S'$ must be isotropic.  Since the rank of
$\hip{}{}$ on $M/\theta M$ is~$1$, there are exactly $8$ possibilities
for $S'$, hence $8$ unimodular Hermitian superlattices $L'$.  One of
these is $L$; write $S$ for its corresponding line in $M/\theta M$.
Note that $S$ contains (the reductions mod~$\theta$ of) no norm~$7$
vectors, because $L$ has no norm~$1$ vectors.  Each $S'\neq S$
contains (the reductions mod~$\theta$ of) at most $6$ norm~$7$
vectors, because $L'$ contains a norm~$1$ vector for every norm~$7$
vector of $M$ projecting into $S'$.  (Each norm~$1$ vector spans a
summand, so $L'$ can have at most $6$ norm~$1$ vectors, and if it has
$6$ then it is a copy of $\curlyO^3$.)  Since $M$ has $42$ norm~$7$ vectors
(lemma~\ref{lem-norm-3-and-7-vectors-in-L}), the subspaces $S'\neq S$
have on average $42/7=6$ (images of) norm~$7$ vectors.  It follows
that each has exactly~$6$, and that each $L'\neq L$ is a copy of $\curlyO^3$.
\end{proof}

\section{The isometry groups $\GAK$ and $\GMu$ of $L$ and $M$ over $\Zhalf$}
\label{sec-Gamma-L-and-Gamma-M}

In this section we regard the isometry groups of $L$ and $M$ as group
schemes over $\Z$ and study the groups $\GAK$ and $\GMu$ of points
over $\Zhalf$.  
$\GMu$ is Mumford's $\Gamma_{\!1}$.
We
show that  $\PGAK$ and $\PGMu$ are densest
possible lattices in $\PGL_3(\Q_2)$, and find generators for them.
We also show by an independent argument that
$\GAK\cap\GMu$ has index~$8$ in  $\GAK$ and $\GMu$.

Regarding $\Isom L$ and $\Isom M$ as group schemes amounts to the
following.  If $A$ is any commutative ring then the $A$-points of
$\Isom L$ are the $(\curlyO\tensor_\Z A)$-linear transformations of
$L\tensor_\Z A$ that preserve the unique $(\curlyO\tensor_\Z
A)$-sesquilinear extension of $\hip{}{}$.  And similarly for $\Isom
M$.  The groups of $\Z$-points of these group schemes are the finite
groups  $\Isom L\iso L_3(2)\times2$ and $\Isom M\iso
F_{21}\times 2$ from the previous section.  We define $\GAK$ and $\GMu$
to be their groups of $\Zhalf$-points.  By the theory of arithmetic
groups \cite[p.~1]{Margulis}, their central quotients $\PGAK$ and $\PGMu$ are
lattices in $\PGL_3(\Q_2)$.  Actually all we need from the general
theory is discreteness, which is the easy part; cocompactness is
part of theorems~\ref{thm-PGamma-M} and~\ref{thm-PGamma-L}.

We write $\Lhalf$ for the Hermitian $\curlyOhalf$-lattice $L\tensor_{\curlyO}\curlyOhalf$ and similarly for $\Mhalf$.
It is easy to find extra isometries of $\Lhalf$ and $\Mhalf$.  For $L$
we observe that the roots $(2,0,0)$, $(0,\lambdabar,\lambdabar)$ and
$(0,\lambdabar,-\lambdabar)$ are mutually orthogonal.  Therefore the
isometry sending the roots $(2,0,0)$, $(0,2,0)$, $(0,0,2)$ of the
standard frame to them lies in $\GAK$.  We call this transformation
$\beta_L$, namely
$$
\beta_L=\begin{pmatrix}
1&0&0\\
0&1/\lambda&1/\lambda\\
0&1/\lambda&-1/\lambda\\
\end{pmatrix}
$$ Similarly, the cyclic permutation $e_1\to e_2\to e_3\to e_1$ does
not quite preserve inner products in $M$, because $\hip{e_2}{e_1}$ and
$\hip{e_3}{e_1}$ are $\lambda$ not $\lambdabar$.  However, this can be
fixed up by multiplying by $\lambda/\lambdabar$.  That is, $\GMu$
contains the isometry $\beta_M:e_1\to e_2\to e_3\to
(\lambda/\lambdabar)e_1$.  This is a slightly more convenient isometry
than Mumford's $\rho$ \cite[p.~241]{Mumford}, and appears in 
\cite[p.~639]{Kato-Shimura-variety}.  Its cube is obviously
the scalar $\lambda/\lambdabar$.

The next two theorems refer to the normalization of Haar measure on
$\PGL_3(\Q_2)$ introduced in section~\ref{sec-finite-and-discrete-subgroups}, namely the one for which
$\PGL_3(\Z_2)$ has mass~$1$.

\begin{theorem}
\label{thm-PGamma-M}
$\PGMu$ is a lattice in $\PGL_3(\Q_2)$ of covolume~$1/21$ and is
generated by $F_{21}$ and $\beta_M$.  It acts transitively on the
vertices, edges and $2$-simplices of $\B$, with stabilizers $F_{21}$,
$\Z/3$ and $\Z/3$ respectively.
\end{theorem}

\begin{proof}
Consider the vertex $v$ of $\B$ corresponding to
$M_2:=M\tensor_\curlyO\Z_2$. We know that $F_{21}$ acts on $v$'s neighbors with
two orbits, corresponding to the points and lines of $\P(M/\lambda
M)=\P(M_2/2M_2)$.  On the other hand, $\beta_M$ cyclically permutes the
three (scalar classes of) $\Z_2$-lattices
$$
\spanof{e_1,e_2,e_3},
\quad
\spanof{e_1,e_2,\frac{1}{\lambda}e_3},
\quad
\spanof{e_1,\frac{1}{\lambda}e_2,\frac{1}{\lambda}e_3}.
$$ This is a rotation of order~$3$ around the center of a $2$-simplex
containing $v$.  It follows that $v$ is
$\generate{F_{21},\beta_M}$-equivalent to each of its neighbors.  Since
the same holds for the neighbors, induction proves
transitivity on vertices.  Transitivity on edges is also clear, and
transitivity on $2$-simplices follows from $F_{21}$'s transitive
action on the edges in the link of $v$.

We have shown that $\generate{F_{21},\beta_M}$ has covolume${}\leq1/21$.
Since $\PGMu$ contains it and has covolume${}\geq1/21$ by
lemma~\ref{lem-second-bounds-on-lattices-in-PGL3Q2}, these groups coincide and have covolume
exactly $1/21$.  That the vertex stabilizers are no larger than
$F_{21}$ also follows from that lemma, and the structures of the other
stabilizers follow.
\end{proof}

\begin{theorem}
\label{thm-PGamma-L}
$\PGAK$ is a lattice in $\PGL_3(\Q_2)$ of covolume $1/21$ and is
generated by $L_3(2)$ and $\beta_L$.  It acts with two orbits on vertices
of $\B$, with stabilizers $L_3(2)$ and $S_4$.
\end{theorem}

\begin{figure}
\psset{unit=.02cm}
\newcount\Ax
\newcount\Ay
\newcount\Bx
\newcount\By
\newcount\Cx
\newcount\Cy
\newcount\Dx
\newcount\Dy
\newcount\Ex
\newcount\Ey
\newcount\Fx
\newcount\Fy
\newcount\Gx
\newcount\Gy
\Ax=0
\Ay=300
\def\ADx{200}
\def\ADy{-50}
\Dx=\Ax
\Dy=\Ay
\advance\Dx by\ADx
\advance\Dy by\ADy
\Gx=\Dx
\Gy=\Dy
\advance\Gx by\ADx
\advance\Gy by\ADy
\def\ABx{50}
\def\ABy{-140}
\Bx=\Ax
\By=\Ay
\advance\Bx by\ABx
\advance\By by\ABy
\Fx=\Bx
\Fy=\By
\advance\Fx by\ADx
\advance\Fy by\ADy
\def\BCx{80}
\def\BCy{200}
\Cx=\Bx
\Cy=\By
\advance\Cx by\BCx
\advance\Cy by\BCy
\def\FEx{-70}
\def\FEy{-80}
\Ex=\Fx
\Ey=\Fy
\advance\Ex by\FEx
\advance\Ey by\FEy
%
%
%
\newcount\tempx
\newcount\tempy
\begin{pspicture}(-100,5)(470,380)
\pspolygon(\Dx,\Dy)(\Ex,\Ey)(\Fx,\Fy)(\Dx,\Dy)
\tempx=\Bx
\tempy=\By
\advance\tempx by\FEx
\advance\tempy by\FEy
\pspolygon(\Ax,\Ay)(\Bx,\By)(\tempx,\tempy)(\Ax,\Ay)
\tempx=\Bx
\tempy=\By
\advance\tempx by-\ADx
\advance\tempy by-\ADy
\pspolygon[fillstyle=solid,linecolor=white](\Bx,\By)(\Ax,\Ay)(\tempx,\tempy)(\Bx,\By)
\pspolygon[fillstyle=solid](\Bx,\By)(\Dx,\Dy)(\Fx,\Fy)(\Bx,\By)
\qline(\Ax,\Ay)(\Gx,\Gy)
\pspolygon[fillstyle=solid](\Bx,\By)(\Cx,\Cy)(\Dx,\Dy)(\Bx,\By)
\tempx=\Cx
\tempy=\Cy
\advance\tempx by\ADx
\advance\tempy by\ADy
\pspolygon[fillstyle=solid](\Fx,\Fy)(\tempx,\tempy)(\Gx,\Gy)(\Fx,\Fy)
\qline(\Ax,\Ay)(\Bx,\By)
\def\solidpercent{40}
\def\dashedpercent{30}
\tempx=-\ADx
\tempy=-\ADy
\multiply\tempx by\dashedpercent
\multiply\tempy by\dashedpercent
\divide\tempx by 100
\divide\tempy by 100
\advance\tempx by\Ax
\advance\tempy by\Ay
\psline[linestyle=dashed](\Ax,\Ay)(\tempx,\tempy)
\tempx=\ADx
\tempy=\ADy
\multiply\tempx by\dashedpercent
\multiply\tempy by\dashedpercent
\divide\tempx by 100
\divide\tempy by 100
\advance\tempx by\Gx
\advance\tempy by\Gy
\psline[linestyle=dashed](\Gx,\Gy)(\tempx,\tempy)
\tempx=\dashedpercent
\advance\tempx by\solidpercent
\tempy=\tempx
\multiply\tempx by-\ADx
\multiply\tempy by-\ADy
\divide\tempx by 100
\divide\tempy by 100
\advance\tempx by \Bx
\advance\tempy by \By
\psline[linestyle=dashed](\Bx,\By)(\tempx,\tempy)
\tempx=\solidpercent
\tempy=\solidpercent
\multiply\tempx by-\ADx
\multiply\tempy by-\ADy
\divide\tempx by 100
\divide\tempy by 100
\advance\tempx by \Bx
\advance\tempy by \By
\psline(\Bx,\By)(\tempx,\tempy)
\tempx=\dashedpercent
\advance\tempx by\solidpercent
\tempy=\tempx
\multiply\tempx by\ADx
\multiply\tempy by\ADy
\divide\tempx by 100
\divide\tempy by 100
\advance\tempx by \Fx
\advance\tempy by \Fy
\psline[linestyle=dashed](\Fx,\Fy)(\tempx,\tempy)
\tempx=\solidpercent
\tempy=\solidpercent
\multiply\tempx by\ADx
\multiply\tempy by\ADy
\divide\tempx by 100
\divide\tempy by 100
\advance\tempx by \Fx
\advance\tempy by \Fy
\psline(\Fx,\Fy)(\tempx,\tempy)
\def\noderadius{5}
\qdisk(\Ax,\Ay){\noderadius}
\qdisk(\Bx,\By){\noderadius}
\qdisk(\Cx,\Cy){\noderadius}
\qdisk(\Dx,\Dy){\noderadius}
\qdisk(\Ex,\Ey){\noderadius}
\qdisk(\Fx,\Fy){\noderadius}
\qdisk(\Gx,\Gy){\noderadius}
\def\myput#1#2#3#4#5{\rput[#1](#2,#3){\raise#5\hbox{#4}}}
\myput{t}{\Ax}{\Ay}{$A\,\,\,\,$}{-15pt}
\myput{br}{\Bx}{\By}{$B\,\,\,$}{6pt}
\myput{t}{\Cx}{\Cy}{$C$}{-22pt}
\myput{t}{\Dx}{\Dy}{$D\,\,\,\,$}{-15pt}
\myput{tl}{\Ex}{\Ey}{$\!E$}{-13pt}
\myput{br}{\Fx}{\Fy}{$F\,\,\,$}{6pt}
\myput{t}{\Gx}{\Gy}{$G\,\,\,\,$}{-15pt}
\def\mymatrix#1#2#3{\tiny\smaller\begin{tabular}{@{}c@{}}#1\\#2\\#3\end{tabular}}
\myput{bl}{\Ax}{\Ay}{\mymatrix{200}{010}{001}}{12pt}
\myput{tl}{\Bx}{\By}{\,\mymatrix{200}{021}{001}}{-18pt}
\myput{r}{\Cx}{\Cy}{\mymatrix{210}{011}{001}\kern6pt}{0pt}
\myput{bl}{\Dx}{\Dy}{\mymatrix{100}{010}{001}}{12pt}
\myput{r}{\Ex}{\Ey}{\mymatrix{201}{021}{001}\,\,\,}{0pt}
\myput{tl}{\Fx}{\Fy}{\mymatrix{100}{021}{001}}{-18pt}
\myput{bl}{\Gx}{\Gy}{\mymatrix{100}{020}{002}}{12pt}
\def\rowvector#1#2#3{\tiny\smaller{}[#1#2#3]}
\def\colvector#1#2#3{$\Bigl[\raise1pt\hbox{\mymatrix{#1}{#2}{#3}}\Bigr]$}
\myput{tr}{\Ax}{\Ay}{\rowvector{1}{0}{0}\kern5pt}{-5pt}
\myput{tr}{\Bx}{\By}{\colvector{0}{1}{1}\kern-3pt}{-28pt}
\myput{l}{\Cx}{\Cy}{\kern5pt\rowvector{1}{1}{1}}{0pt}
\myput{l}{\Ex}{\Ey}{\kern10pt\colvector{1}{1}{1}}{0pt}
\myput{l}{\Fx}{\Fy}{\kern10pt\rowvector{0}{1}{1}}{6pt}
\myput{tl}{\Gx}{\Gy}{\colvector{1}{0}{0}}{-20pt}
\end{pspicture}
\caption{The subcomplex $X$ of $\B$ used in the proof of
  theorem~\ref{thm-PGamma-L}; $\beta_L$ acts by rotating about the
  centerline of the strip and sliding everything to the left.}
\label{fig-subcomplex-of-building}
\end{figure}
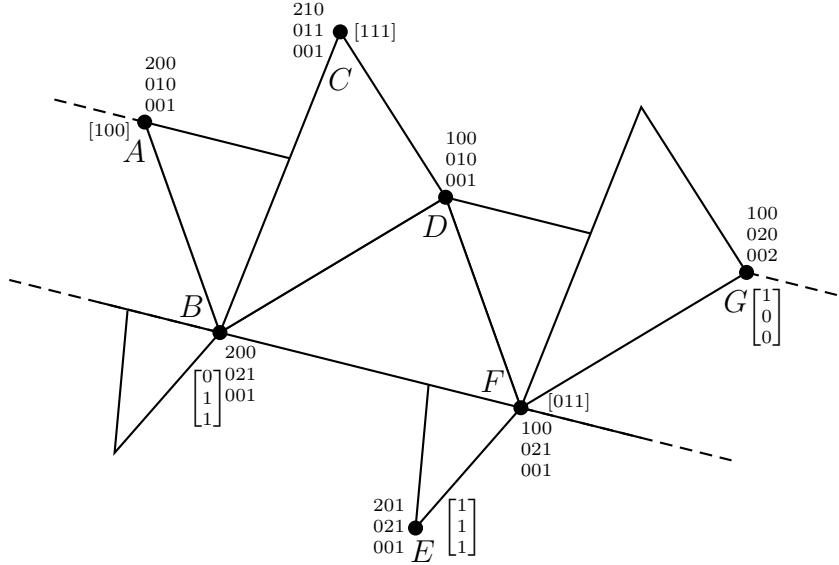

\begin{proof}
The main step is to show that $\generate{L_3(2),\beta_L}$ acts
with${}\leq2$ orbits on vertices.  Consider the subcomplex $X$ of $\B$
pictured in figure~\ref{fig-subcomplex-of-building}.  (It is the
fixed-point set of the dihedral group $D_8\sset L_3(2)$ generated by
the negations of evenly many coordinates, together with the
simultaneous negation of the first coordinate and exchange of the
second and third.)  We have named seven vertices $A,\dots,G$ and given
bases for the $\Z_2$-lattices in $\Q_2^3$ they represent (the columns
of the $3\times3$ arrays).  Obviously $D$ corresponds to $\Z_2^3$, and
one can check that $C$ represents $L\tensor_\curlyO\Z_2$.  For
$A,B,C,E,F,G$ we have also indicated their positions in the link of
$D$ by giving the corresponding subspace of
$\Z_2^3/2\Z_2^3\iso\F_2^3$.  A column vector represents its span and a
row vector represents its kernel (thinking of it as a a linear function).

One can check that $\beta_L$ acts on $X$ by rotating everything around
the centerline of the
main strip and shifting to the left by half a notch.  In
particular, the vertices along the edges of the strip are all
$\generate{\beta_L}$-equivalent, as are the tips of the triangular
flaps.  We claim that every vertex of
$\B$ is $\generate{L_3(2),\beta_L}$-equivalent to $C$ or $D$.  It
suffices by induction to verify this for every vertex adjacent to $C$ or $D$.  Any
neighbor of $C$ is $L_3(2)$-equivalent to $B$ or $D$, and we
just saw that $\beta_L(D)=B$.  The stabilizer of
$D$ in $L_3(2)$ is the subgroup $S_4$ preserving the standard frame in
$L$.  It acts on $\F_2^3$ by the $S_3$ of coordinate permutations.  The $14$ neighbors of $D$ correspond
to the nonzero row and column vectors over $\F_2$.  Under the $S_4$ symmetry,
every neighbor is equivalent to $A$, $B$, $C$, $E$,
$F$ or $G$.  Since $\beta_L(E)=C$ and $A$, $B$, $F$ and $G$ are
$\generate{\beta_L}$-equivalent to $D$, we have proven our claim.

Since $\generate{L_3(2),\beta_L}$ has covolume${}\leq1/168+1/24=1/21$ and
$\PGAK$ has covolume${}\geq1/21$
(lemma~\ref{lem-second-bounds-on-lattices-in-PGL3Q2}), these two
groups coincide.  Lemma~\ref{lem-second-bounds-on-lattices-in-PGL3Q2}
also implies that the stabilizers of $C$ and $D$ are no larger than
the visible $L_3(2)$ and $S_4$.
\end{proof}

A key relation between $\GAK$ and $\GMu$ is the following.  It is
independent of the other theorems in this section.  In fact, together
with either of theorems~\ref{thm-PGamma-M} and~\ref{thm-PGamma-L} it
implies the other, except for the explicit generating sets.

\begin{theorem}
\label{thm-intersection-has-index-8-in-each}
$\GAK\cap\GMu$ has index $8$ in each of $\GAK$ and $\GMu$.
\end{theorem}

\begin{proof}
  This boils down to two claims: $\Lhalf$ contains exactly $8$ copies
  of $\Mhalf$, on which $\GAK$ acts transitively, and $\Mhalf$ lies in
  exactly $8$ copies of $\Lhalf$, on which $\GMu$ acts transitively.

  For the first claim, any sublattice of $\Lhalf$ isometric to
  $\Mhalf$ must have index~$7$ since $\det M=7$ and $L$ is unimodular.
  The index~$7$ sublattices correspond to the hyperplanes in
  $\Lhalf/\theta\Lhalf=L/\theta L$, which we studied in
  section~\ref{sec-L-and-M}.  The orthogonal complements of the plus
  and minus points of $\P(L/\theta L)$ give $\curlyO$-sublattices $M'$ with
  $M'[\frac{1}{2}]$ not isomorphic to $\Mhalf$.  (The reduction of
  $\hip{}{}$ to $M'/\theta M'$ has rank~$2$ not~$1$.)  Since
  $L_3(2)\sset\GAK$ acts transitively on the $8$ isotropic points of
  $\P(L/\theta L)$, the claim is proven.

  The second claim is similar.  By
  lemma~\ref{lem-basic-properties-of-M}, the unimodular lattices
  containing $\Mhalf$ are $\Lhalf$ and $7$ copies of $\curlyOhalf^3$.
  Happily, from the definition of $L$ in section~\ref{sec-L-and-M} it is
  obvious that $\Lhalf=\curlyOhalf^3$.  To prove transitivity, suppose
  $\Lhalf'$ is one of $7$, and choose any isometry carrying it to
  $\Lhalf$.  This sends $\Mhalf$ to one of $8$ sublattices of
  $\Lhalf$.  Following this by an isometry of $\Lhalf$ sending this
  image of $\Mhalf$ back to $\Mhalf$, we see that $\Lhalf'$ is
  $\GMu$-equivalent to $\Lhalf$.
\end{proof}



\section{Uniqueness}
\label{sec-uniqueness}

In this section we show that $\PGAK$ and $\PGMu$ are the only two densest
lattices in $\PGL_3(\Q_2)$.  We will use the fact that $\PGL_3(\Q_2)$
contains only one conjugacy classes of subgroups isomorphic to
$F_{21}$ resp.\ $L_3(2)$.     One way to see this is
that the finite group has only two faithful characters of
degree $3$, which are exchanged by an outer automorphism.

\begin{theorem}
\label{thm-uniqueness}
Suppose $\PG$ is a lattice in $\PGL_3(\Q_2)$ with smallest possible
covolume.  Then its covolume is $1/21$ and it is conjugate to
$\PGAK$ or $\PGMu$.
\end{theorem}

\begin{proof}
The covolume claim is proven in
lemma~\ref{lem-second-bounds-on-lattices-in-PGL3Q2}, which also shows
that $\PG$ either acts transitively on vertices of $\B$ with
stabilizer isomorphic to $F_{21}$, or has two orbits, with stabilizers
isomorphic to $L_3(2)$ and $S_4$.

We begin by assuming the latter case and proving $\PG$ conjugate to
$\PGAK$.  As in the proof of theorem~\ref{thm-PGamma-L}, let $C$ be
the vertex in $\B$ corresponding to $L$ and consider its neighbors $B$
and $D$ and the transformation $\beta_L\in\GAK$ sending $D$ to $B$.
Because $\PGL_3(\Q_2)$ contains a unique conjugacy class of
$L_3(2)$'s, we may suppose that $\PG$ contains the $\PGAK$-stabilizer
$L_3(2)$ of $C$.  By lemma~\ref{lem-L32-actions-on-neighbors}, the
$\PG$-stabilizers of $B$ and $D$ are subgroups $S_4$ of this $L_3(2)$.
By the assumed transitivity of $\PG$ on vertices with stabilizer
$S_4$, it contains some $\beta\in\PGL_3(\Q_2)$ sending $D$ to $B$.  So
$\beta\circ\beta_L^{-1}$ normalizes the $S_4$ fixing $B$.  Now, this $S_4$
is self-normalizing in $\PGL_3(\Q_2)$ since $S_4$ has no outer
automorphisms and the centralizer is trivial (by irreducibility).  So
$\beta$ and $\beta_L$ differ by an element of this $S_4$.  In particular,
$\PG$ contains $\beta_L$.  By theorem~\ref{thm-PGamma-L} we have
$\PGAK=\generate{L_3(2),\beta_L}$, so we have shown that $\PG$ contains
$\PGAK$.  By the maximality of the latter group, we have equality.

The case of $\PG$ transitive on vertices with stabilizer $F_{21}$ is
similar in spirit, but messier because the inclusions $S_4\sset
L_3(2)$ are replaced by $\Z/3\sset F_{21}$, and $\Z/3$ is very far
from self-normalizing in $\PGL_3(\Q_2)$.  Before embarking on the
details, we observe that each triangle in $\B$ has $\PG$-stabilizer
$\Z/3$, acting on it by cyclically permuting its vertices.  This is a
restatement of the simple transitivity of $\PG$ on pairs (vertex of
$\B$, triangle containing it), which can be checked by considering the
action of $F_{21}$ on the link of its fixed vertex.

Write $v$ for the vertex of $\B$ corresponding to
$M_2:=M\tensor\Z_2$, preserved by the group $F_{21}\times2$ of $\Z$-points of $\GMu$.  
We will need Mumford's explicit
generators
$$
\sigma=\begin{pmatrix}1&0&\lambda\\0&0&-1\\0&1&1\end{pmatrix}
\quad\hbox{and}\quad
\tau=\begin{pmatrix}0&0&1\\1&0&1+\lambda\\0&1&\lambda\end{pmatrix}
$$
of orders $3$ and~$7$ respectively.  These matrices come from
Mumford's description of $M$ as a Hermitian $\curlyO$-lattice in 
the  $\Q(\lambda)$-vector space $\Q(\zeta_7)$, with basis
$1,\zeta_7,\zeta_7^2$.  Namely, $\tau$ is multiplication by $\zeta_7$
and $\sigma$ is the Galois automorphism $\zeta_7\to\zeta_7^2$.  As a
$3$-element of $L_3(2)$, $\sigma$ preserves a unique point and
a unique line of $\P(M_2/2M_2)$.  We
write $p$ and $l$ for the corresponding neighbors of $v$.  By working
out the action of $\tau$ on $\P(M_2/2M_2)$, one can check that
$\tau^3(l)$ is a neighbor of $p$.  

Since $\PGL_3(\Q_2)$ contains a unique conjugacy class of $F_{21}$'s,
we may suppose $\PG$ contains the $\PGMu$-stabilizer of $v$.  Let
$\G_0\sset\GL_3(\Q_2)$ be generated by $\sigma$, $\tau$ and an
$\a\in\GL_3(\Q_2)$ lying over the element of $\PG$ that  rotates the triangle with vertices
$v,p,\tau^3(l)$ as shown:
\begin{equation}
\psset{unit=1.5cm,linewidth=.03}
\def\noderadius{.07}
\begin{pspicture}[shift=-.55](-1.3,-.1)(1,1)
\qdisk(-1,0){\noderadius}
\qdisk(0,0){\noderadius}
\qdisk(1,0){\noderadius}
\qdisk(.5,.866){\noderadius}
\qline(-1,0)(1,0)
\qline(0,0)(.5,.866)
\qline(1,0)(.5,.866)
\rput[r](-1.2,0){$l$}
\rput[br](0,.1){$v$}
\rput[l](1.2,0){$p$}
\rput[l](.7,.866){$\tau^3(l)$}
\psset{linewidth=.02}
\def\arcradius{.22}
\rput[c](.5,.289){$\a$}
\psarc{->}(.5,.289){\arcradius}{90}{205}
\psarc{->}(.5,.289){\arcradius}{210}{325}
\psarc{->}(.5,.289){\arcradius}{330}{85}
\end{pspicture}
\label{eq-diagram}
\end{equation}
By our remarks above on the $\PG$-stabilizer of a $2$-simplex, $\a$ is
uniquely defined (up to scalars) and its cube is a scalar.
We will replace $\a$ by its product with a scalar whenever convenient.
Note that $\c:=\a\tau^3$ sends the edge $\overline{lv}$ to
$\overline{vp}$.  Since these two edges have the same $\PG$-stabilizer
$\generate{\sigma}$, $\c$ normalizes $\generate{\sigma}$.  We will consider the
case that $\c$ centralizes $\sigma$ and then the case that it inverts
$\sigma$. 

The centralizer of $\sigma$ in $\GL_3(\Q_2)$ is the product of the
scalars and
$$
\set{\pi_1+a\pi_2+b\sigma_2}{\hbox{$a,b\in\Q_2$, not both $0$}}
\sset\GL_3(\Q_2)
$$
where $\pi_1=\frac{1}{3}(1+\sigma+\sigma^2)$ is the projection to
$\sigma$'s fixed space, $\pi_2=I_3-\pi_1$ is the projection to the
span of $\sigma$'s other eigenspaces, and
$\sigma_2=\sigma\circ\pi_2$.  This is because the image of $\pi_2$ is
irreducible as a $\Q_2[\generate{\sigma}]$-module, making it into a
$1$-dimensional vector space over $\Q_2(\zeta_3)$.  Expressing
$\a$ in terms of $\c$, we get $\a=(\pi_1+a\pi_2+b\sigma_2)\tau^{-3}$,
and the equation $\a^3=(\hbox{scalar})$ imposes conditions on $a$
and~$b$, namely the vanishing of 8 polynomials in $\Q(\lambda)[a,b]$.  These
are unwieldy enough that we used the PARI/GP software \cite{PARI} to
handle the algebra.  One uses Gaussian elimination on the $a^3$,
$a^2b$, $a^2$, $ab^2$ and $ab$ coefficients, obtaining a relation
$a=$(polynomial in $b$).  After eliminating $a$ in favor
of $b$, one of the relations becomes $b^3-b=0$, so $b\in\{0,\pm1\}$.
Of these, $b=0$ does not satisfy the other relations, while
$b=\pm1$ do.  But
these give the trivial solutions $\c=\sigma^{\pm1}$, and the
corresponding $\a=\sigma^{\pm1}\tau^{-3}$ fix $v$ rather than acting
as in \eqref{eq-diagram}.  So we have eliminated the case that $\c$ centralizes
$\sigma$.

To treat the case of $\c$ inverting $\sigma$, we note that
$\a_M:=\sigma\tau^5\beta_M^{-1}\tau^{-5}\sigma^{-1}$ is an element of
$\GMu$ acting as in \eqref{eq-diagram}, and that the corresponding
$\c_M:=\a_M\tau^3$ inverts $\sigma$.  Therefore $\c$ has the form
$\c_M\circ(\pi_1+a\pi_2+b\sigma_2)$.  As before, the equation
$\a^3=\hbox{(scalar)}$ imposes conditions on $a$ and~$b$.  Gaussian
elimination yields a relation $f(b)a+g(b)=0$ with $f$ and $g$
polynomials and $f$ of degree~$1$.  After checking that $f(b)\neq0$
(it turns out that $f(b)=0$ implies $g(b)\neq0$), one solves for $a$
in terms of $b$.  Then eliminating $a$ gives a family of polynomials
in $b$, all of which must vanish.  It happens that all are divisible
by $b$, so $b=0$ is a solution.  This leads to $a=1$, hence $\a=\a_M$
and $\generate{\sigma,\tau,\beta_M}\sset\G_0$.  By
theorem~\ref{thm-PGamma-M}, the projectivization of the left side is $\PGMu$.
Since $\PG_0\sset\PG\sset\PGMu\sset\PG_0$, all three groups coincide.

If $b\neq0$ then we divide the polynomials by as many powers of $b$ as
possible and take their gcd in $\Q(\lambda)[b]$, which turns out to
have degree~$1$.  Solving for $b$, one obtains $\a$.  It turns out
that this $\a$ is a scalar times a matrix with entries in $\curlyO$
and odd determinant.  That is, it represents an element of
$\PGL(M_2)$.  Therefore it fixes $v$ rather than sending it to $p$, so
this solution for $b$ is spurious.  This completes the proof.
\end{proof}

\end{document}